\documentclass[11pt,oneside]{amsart}
\usepackage{amssymb,latexsym,amsthm}
\usepackage[bookmarksopen=true]{hyperref}
\usepackage[dvips]{graphicx}
\DeclareGraphicsExtensions{.bmp,.png,.pdf,.jpg}
\usepackage{color}
\usepackage{pstricks-add}

\newcommand{\bea}{\begin{eqnarray}}
\newcommand{\eea}{\end{eqnarray}}

\usepackage{color}
\usepackage{graphicx}
\usepackage{epsfig}
\usepackage{float}
\usepackage[english]{babel}
\newcommand{\R}{\mathbb R}

\newcommand{\D}{\displaystyle}
\newcommand{\Div}{\mathrm{div}}
\newcommand{\Od}{\Omega\setminus\overline{D}}

\newcommand{\vp}{\varphi}
\newcommand{\Om}{\Omega}
\newcommand{\pa}{\partial}

\newtheorem{theorem}{Theorem}[section]
\newtheorem{lemma}[theorem]{Lemma}
\newtheorem{proposition}[theorem]{Proposition}
\newtheorem{remark}[theorem]{Remark}
\newtheorem{corollary}[theorem]{Corollary}
\newtheorem{definition}[theorem]{Definition}

\title{Size estimates of an obstacle in a stationary Stokes fluid}

\author{E. Beretta$^1$, C. Cavaterra$^{2}$, J. H. Ortega$^3$ \& S. Zamorano$^{3,4}$}

\begin{document}

\maketitle

\footnotetext[1]{Dipartimento di Matematica, Politecnico di Milano, Milano 20133, Italy \\
E-mail: {\it elena.beretta@polimi.it}}

\footnotetext[2]{Dipartimento di Matematica, Universit\`{a} degli Studi di Milano, Milano 20133, Italy \\
E-mail: {\it cecilia.cavaterra@unimi.it}}

\footnotetext[3]{Centro de Modelamiento Matem\'atico (CMM) and Departamento de Ingenier\'ia Matem\'atica,
Universidad de Chile (UMI CNRS 2807), Avenida Beauchef 851, Ed. Norte, Casilla 170-3, Correo 3, Santiago, Chile \\
E-mail: {\it jortega@dim.uchile.cl}, {\it szamorano@dim.uchile.cl}}

\footnotetext[4]{Basque Center for Applied Mathematics - BCAM, Mazarredo 14, E-48009, Bilbao, Basque Country, Spain}

\begin{abstract}
In this work we are interested in estimating the size of a cavity $D$ immersed in a bounded domain $\Omega\subset \R^d,$ $d=2,3,$ filled with a viscous fluid
governed by the Stokes system, by means of velocity and Cauchy forces on the external boundary $\partial\Omega$.
More precisely, we establish some lower and upper bounds in terms of the difference between the
external measurements when the obstacle is present and without the object. The proof of the result is based on  interior regularity results and quantitative
estimates of unique continuation for the solution of the Stokes system.
\end{abstract}

\smallskip\

\noindent{AMS classification scheme numbers\/} : 35R30, 65M32, 76D07, 76D03

\smallskip\

\noindent{\it Keywords\/} : Inverse Problems, Stokes System, Size Estimate, Interior Regularity, Boundary Value Problems, Numerical Analysis.

\section{\bfseries Introduction}
\
\par

We consider an obstacle $D$ immersed in a region $\Omega\subset \mathbb{R}^d$ $(d=2,3)$ which is filled with a viscous fluid.
Then, the velocity vector $u$ and the scalar pressure $p$ of the fluid in the presence of the obstacle $D$ fulfill the following boundary value
problem for the Stokes system:
\begin{equation}\label{P1}
\left\{\begin{array}{rllll}
-\Div(\sigma(u,p))&=&0& ,& \textrm{ in }\Od,\\
\Div u&=&0&,&\textrm{ in } \Od,\\
u&=&g&,& \textrm{ on }\pa\Om,\\
u&=&0&,&\textrm{ on }\pa D,
\end{array}
\right.
\end{equation}
where $\sigma(u,p)=2\mu e(u)-pI$ is the stress tensor,  $e(u)=\frac{(\nabla u+\nabla u^{T})}{2}$ is the strain tensor, $I$ is the identity matrix of order
$d\times d$, $n$ denotes the exterior unit normal to $\pa\Om$ and $\mu >0$ is the kinematic viscosity. The condition $u|_{\pa D}=0$ is the so called
\emph{no-slip} condition.

Given the boundary velocity $g\in (H^{3/2}(\partial\Omega))^d$ satisfying the compatibility condition
\begin{equation*}
\int_{\pa\Om} g\cdot n =0,
\end{equation*}
we consider the solution to Problem \eqref{P1}, $(u,p)\in (H^1(\Omega\backslash \overline{D}))^d\times L^2(\Omega\backslash \overline{D})$, and measure the corresponding Cauchy force
on $\partial\Omega$, $\D\psi=\sigma(u,p) n|_{\pa\Om}$, in order to recover the  obstacle $D$.
Then, it is well known that this inverse problem has a unique solution. In fact, in \cite{alvarez2005identification}, the authors prove uniqueness
in the case of the steady-state and evolutionary Stokes system using unique continuation property of solutions.
By uniqueness we mean the following fact: if $u_1$ and $u_2$ are two solutions of \eqref{P1} corresponding to a given boundary
data $g$, for obstacles $D_1$ and $D_2$ respectively, and $\sigma(u_1,p_1) n=\sigma(u_2,p_2) n$ on an open subset $\Gamma_0\subset \pa\Om$,
then $D_1=D_2$.  Moreover, in \cite{ballerini2010stable}, $\log-\log$ type stability estimates for the Hausdorff distance between the boundaries
of two cavities in terms of the Cauchy forces have been derived.
Reconstruction algorithms for the detection of the obstacle have been proposed in \cite{CCG2015} and in  \cite{heck2007reconstruction}.
The method used in \cite{heck2007reconstruction} relies on the construction of special complex geometrical optics solutions for the
stationary Stokes equation with a variable viscosity. In \cite{CCG2015}, the detection algorithm is based on topological sensitivity and
shape derivatives of a suitable functional.
We would like to mention that there hold $\log$ type stability estimates for the Hausdorff distance between the boundaries of two cavities in terms of
boundary data, also in the case of conducting cavities and elastic cavities (see \cite{ABRV2001}, \cite{CHY2001} and \cite{MR2004}).
These very weak stability estimates reveal that the problem is severly ill posed limiting the possibility of efficient reconstruction of the
unknown object and motivating mathematically, but also from the point of view of applications, the importance of the identification of partial
information on the unknown obstacle $D$ like, for example, the size.

In literature we can find several results concerning the determination of inclusions or cavities and the estimate of their sizes related to
different kind of models. Without being exhaustive, we quote some of them.
For example in \cite{kang2012sharp} and \cite{kang2013bounds} the problem of estimating the volume of inclusions is analyzed using a finite number
of boundary measurements in electrical impedance tomography.
In \cite{doubova2015some},  the authors prove uniqueness, stability and reconstruction of an immersed obstacle in a system modeled by a linear wave equation.
These results are obtained applying the unique continuation property for the wave equation and in the
two dimensional case the inverse problem is transformed in a well-posed problem for a suitable cost functional.
We can also mention \cite{heck2007reconstruction}, in which it is analyzed the problem of reconstructing obstacles inside a bounded domain
filled with an incompressible fluid by means of special complex geometrical optics solutions for the stationary Stokes equation.

Here we follow the approach introduced by Alessandrini et al. in \cite{alessandrini2002detecting} and in \cite{morassi2003detecting}
and we establish a quantitative estimate of the size of the obstacle D, i.e. $|D|$, in terms of suitable boundary measurements. More precisely,
let us denote by $(u_0,p_0)\in (H^1(\Omega))^d\times L^2(\Omega)$ the velocity vector of the fluid and the pressure in the absence of the obstacle $D$,
namely the solution to the Dirichlet problem
\begin{equation}\label{ec1}
\left\{\begin{array}{rllll}
-\Div(\sigma(u_0,p_0))&=&0& ,& \textrm{ in }\Omega,\\
\Div \ u_0&=&0&,&\textrm{ in } \Omega,\\
u_0&=&g&,& \textrm{ on }\partial\Omega.
\end{array}
\right.
\end{equation}
and let $\D\psi_0=\sigma(u_0,p_0) n|_{\pa\Om}$.
We consider now the following quantities
\begin{equation*}
W_0=\D\int_{\pa\Om}g\cdot \psi_0\; \qquad \textrm{and} \qquad\; W=\D\int_{\pa\Om}g\cdot \psi ,
\end{equation*}
representing the measurements at our disposal. Observe that the following identities hold true
\begin{equation*}
W_0=2\D\int_{\Om}|e(u_0)|^2\; \qquad \textrm{and} \qquad\; W=2\int_{\Omega\backslash\overline D}|e(u)|^2 ,
\end{equation*}
giving us the information on the total deformation of the fluid in the corresponding domains, $\Omega$ and $\Omega \backslash \overline D$.
We will establish a quantitative estimate of the size of the obstacle D, $|D|$, in terms of the difference $W-W_0$. In order to accomplish this
goal, we will follow the main track of \cite{alessandrini2002detecting} and \cite{morassi2003detecting} applying fine interior regularity results,
Poincar\'{e} type inequalities and quantitative estimates of unique continuation for solutions of the stationary Stokes system.
The plan of the paper is as follows. In Section \ref{sec2} we provide the rigorous formulations of the direct problem and state the main results,
Theorems \ref{teo1}-\ref{teo2}. Section \ref{sec3} is devoted to the proofs of Theorems \ref{teo1}-\ref{teo2}. Finally in Section \ref{sec4} we show
some computational examples.

\setcounter{equation}{0}
\section{\bfseries Main results}\label{sec2}
\
\par

In this section we introduce some definitions and some preliminary results we will use through the paper and we will state our main theorems.

Let $x\in\R^d$, we denote by $B_{r}(x)$ the ball in $\R^d$ centered in $x$ of radius $r$. We will indicate by $\cdot$ the scalar product between
vectors or matrices. We set $x=(x_1,\ldots,x_{d})$ as $x=(x',x_{d})$, where $x'=(x_1,\ldots,x_{d-1})$.

\begin{definition}[Def. $2.1$ \cite{alessandrini2002detecting}]

Let $\Om\subset\R^d$ be bounded domain. We say that $\pa\Om$ is of class $C^{k,\alpha},$  with constants $\rho_0,\;M_0>0$, where $k$ is a nonnegative
integer and $\alpha\in[0,1)$, if, for any $x_0\in\pa\Om,$ there exists a rigid transformation of coordinates, in which $x_0=0$ and
\begin{equation*}
\Om\cap B_{\rho_0}(0)=\{x\in B_{\rho_0}(0):\;x_{n}>\vp(x')\},
\end{equation*}
where $\vp$ is a function of class $C^{k,\alpha}(B'_{\rho}(0)), \ k\geq 1$, such that
\begin{equation*}
\vp(0) =0,\ \qquad 
\nabla\vp(0)=0 \ \qquad \textrm{and} \ \qquad
\|\vp\|_{C^{k,\alpha}(B'_{\rho_0}(0))} \leq M_0\rho_0.
\end{equation*}
\end{definition}
When $k=0$ and $\alpha=1$ we will say that $\pa\Om$ is of Lipschitz class with constants $\rho_0,M_0$.

\begin{remark}
We normalize all norms in such a way that they are dimensionally equivalent to their argument, and coincide with the usual norms when $\rho_0=1$.
In this setup, the norm taken in the previous definition is intended as follows:
\begin{equation*}
\|\phi\|_{C^{k,\alpha}(B'_{\rho_0}(0))}=\D\sum_{i=0}^{k}\rho_0^{i}\|D^{i}\phi\|_{L^{\infty}(B_{\rho_0}'(0))}+\rho_0^{k+\alpha}|D^{k}\phi|_{\alpha,B_{\rho_0}'(0)},
\end{equation*}
where $|\cdot|$ represents the $\alpha$-H\"older seminorm
\begin{equation*}
\D |D^{k}\phi|_{\alpha,B_{\rho_0}'(0)}=\sup_{x',y'\in B_{\rho_0}'(0),x'\neq y'}\frac{|D^{k}\phi(x')-D^{k}\phi(y')|}{|x'-y'|^{\alpha}},
\end{equation*}
and $D^{k}\phi=\{D^{\beta}\phi\}_{|\beta|=k}$ is the set of derivatives of order $k$. Similarly we set the norms
\begin{equation*}
\D\|u\|_{L^2(\Om)}^2=\D\frac{1}{\rho_0^{d}}\int_{\Om}|u|^2\ \quad  \textrm{and} \ \quad
\D\|u\|_{H^1(\Om)}^2=\D \frac{1}{\rho_0^{d}}\left(\int_{\Om}|u|^2+\rho_0^2\int_{\Om}|\nabla u|^2\right).
\end{equation*}
\end{remark}

\subsection{Some classical results for Stokes problem}
\
\par

We now define the following quotient space since, if we consider incompressible models, the pressure is defined
only up to a constant.

\begin{definition}
Let $\Om$ be a bounded domain in $\R^d$. We define the quotient space
\begin{equation*}
L_{0}^{2}(\Om)=L^2(\Om)/\R,
\end{equation*}
 represented by the class of functions of $L^2(\Om)$ which differ by an additive constant.
We equip this space with the quotient norm
\begin{equation*}
\D\|v\|_{L_0^2(\Om)}=\inf_{\alpha\in\R}\|v+\alpha\|_{L^2(\Om)}.
\end{equation*}
\end{definition}

The Stokes problem has been studied by several authors and, since it is impossible to quote all the related relevant contributions,  we refer
the reader to the extensive surveys \cite{girault2012finite} and \cite{temam2001navier}, and the references therein.
We limit ourselves to present some classical results, useful for the treatment of our problem, concerning existence, uniqueness, stability
and regularity of solutions to the following boundary value problem
for the Stokes system
\begin{equation}\label{2.1}
\left\{\begin{array}{rllll}
-\Div(\sigma(u,p))&=&f& ,& \textrm{ in }\Om,\\
\Div \, u &=&0&,&\textrm{ in } \Om,\\
u&=&g&,& \textrm{ on }\pa\Om,
\end{array}
\right.
\end{equation}
where, for the sake of simplicity, from now on we assume $\mu(x)\equiv 1$, $\forall x\in\Omega$.

Concerning the well-posedness of this problem we have
\begin{theorem}[Existence and uniqueness, \cite{temam2001navier}] Let $\Om\subset\R^d$ be a bounded domain of class $C^2$, with $d\geq 2$.
Let $f\in (H^{-1}(\Om))^{d}$ and $g\in (H^{1/2}(\pa\Om))^{d}$ satisfying the compatibility condition
\begin{equation}\label{2.2}
\D\int_{\pa\Om}g\cdot n =0.
\end{equation}
Then, there exists a unique $(u,p)\in ((H^1(\Om))^{d}\times L_0^{2}(\Om))$ solution to problem \eqref{2.1}. Moreover, there exists a positive constant
$C$, depending only on $\Om$, such that
\begin{equation*}
\|u\|_{H^1(\Om)}+\|p\|_{L_0^2(\Om)}\leq C(\|f\|_{H^{-1}(\Om)}+\|g\|_{H^{1/2}(\partial\Om)}).
\end{equation*}
\end{theorem}
Regarding the regularity, the following result holds
\begin{theorem}[Regularity of the Stokes problem, \cite{temam2001navier}]\label{TSR} Let $\Om$ be a bounded domain of class $C^{k+1,1}$ in $\R^d$,
with $k \in \mathbb{N} \cup \{0\}$ and $d \geq 2$. Then, for any $f\in (H^{k}(\Om))^{d}$ and $g\in (H^{k+3/2}(\pa\Om))^{d}$ satisfying \eqref{2.2},
the unique solution to \eqref{2.1} is such that
\begin{equation*}
(u,p)\in (H^{k+2}(\Om))^d\times H^{k+1}(\Om).
\end{equation*}
Moreover, we have
\begin{equation*}
\|u\|_{H^{k+2}(\Om)}+\|p\|_{H^{k+1}(\Om)}\leq C(\|f\|_{H^{k}(\Om)}+\|g\|_{H^{k+3/2}(\partial\Om)}),
\end{equation*}
where $C$ is a positive constant depending only on $\Om$.
\end{theorem}

\smallskip

\subsection{Preliminaries}
\
\par

In order to prove our main results we need the following a-priori assumptions on $\Om$,  $D$ and the boundary data $g$.
\begin{enumerate}
\item[\bfseries{(H1)}] $\Om\subset\R^d$ is a bounded domain with a connected boundary $\pa\Om$ of Lipschitz class with constants $\rho_0,M_0$.
Further, there exists $M_1>0$ such that
\begin{equation}\label{2}
|\Om|\leq M_1\rho_{0}^{d}.
\end{equation}

\item[\bfseries{(H2)}] $D\subset\Om$ is such that $\Od$ is connected and it is strictly contained in $\Om$, that is there exists a positive constant
$d_0$ such that
\begin{equation}\label{H2}
d(D,\pa\Om)\geq d_0 >0.
\end{equation}
Moreover, $D$ has a connected boundary
$\pa D$ of class $C^{2,\alpha}$, $\alpha \in (0,1]$, with constants $\rho,L$.
\item[\bfseries{(H3)}]
$D$ satisfies $({\bf H2})$ and the scale-invariant fatness condition with constant $Q>0$, that is
\begin{equation}\label{3}
diam(D)\leq Q\rho.
\end{equation}
\item[\bfseries{(H4)}] $g$ is such that
\begin{equation*}
g\in (H^{3/2}(\pa\Om))^{d},\quad g\not\equiv 0,\quad
\D\frac{\|g\|_{H^{1/2}(\pa\Om)}}{\|g\|_{L^{2}(\pa\Om)}}\leq c_0,
\end{equation*}
for a given constant $c_0>0$,
and satisfies the compatibility condition
\begin{equation*}
\D\int_{\pa\Om}g\cdot n =0.
\end{equation*}
Also suppose that there exists a point $P\in\pa\Om,$ such that,
\begin{equation*}
g=0\textrm{ on }\pa\Om\cap B_{\rho_0}(P).
\end{equation*}
\item[\bfseries{(H5)}] Since one measurement $g$ is enough in order to detect the size of $D$, we choose  $g$ in such a way
that the corresponding solution $u$ satisfies the following condition
\begin{equation}\label{bg}
\D\int_{\pa \Om}\sigma(u,p)n=0.
\end{equation}
\end{enumerate}
Concerning assumption {\bfseries(H5)}, the following result holds.
\begin{proposition}
There exists at least one function $g$ satisfying $({\bf H4})$ and $(\bf {H5})$.
\end{proposition}

\begin{proof} Consider $(d+1)$ linearly independent functions $g_{i}$ satisfying $({\bf H4})$, $i=1,\ldots,d+1$.

Let
\begin{equation*}
\D\int_{\pa\Om}\sigma(u_{i},p_{i})n=v_{i}\in\R^d,
\end{equation*}
where $(u_{i},p_{i})$ is the corresponding solution of \eqref{P1} associated to $g_{i}$, $i=1,\ldots,d+1$.

If, for some $i$, we have that $v_{i}=0$, then the result follows. So, assume that all the $v_{i}$ are different from the null vector.
Then, there exist some constants $\lambda_{i}$, with $i=1,\ldots,d+1$, not all zero,
such that
\begin{equation*}
\D\sum_{i=1}^{d+1}\lambda_{i}v_{i}=0
\end{equation*}
and we can choose our Dirichlet boundary data as
\begin{equation*}
g=\D\sum_{i=1}^{d+1}\lambda_{i}g_{i}.
\end{equation*}
Therefore, $g$ satisfies $({\bf H4})$ and since the Cauchy force is linear with respect to the Dirichlet boundary condition we have
\begin{equation*}
\D\int_{\pa\Om}\sigma(u,p)n=0,
\end{equation*}
where $(u,p)$ is the corresponding solution to \eqref{P1}, associated to $g$.
\end{proof}

\begin{remark}
Integrating the first equation of \eqref{P1} on $\Od$, applying the Divergence Theorem and using (\ref{bg}), we obtain
\begin{equation}\label{bgob}
\int_{\pa D}\sigma(u,p)n = 0.
\end{equation}
\end{remark}

\begin{remark}\label{R1}
Notice that the constant $\rho$ in $({\bf H2})$ already incorporates information on the size of $D$. In fact, an easy computation shows that if $D$
has a boundary of class $C^{2,\alpha}$ with constant $\rho$ and $L$, then we have
\begin{equation*}
|D|\geq C(L)\rho^{d}.
\end{equation*}
Moreover, if also condition $({\bf H3})$ is satisfied, then it holds
\begin{equation*}
|D|\leq C(Q)\rho^{d}.
\end{equation*}
\end{remark}

\begin{remark}\label{R1bis} If $D$ satisfies $({\bf H2})$, then there exists a constant $h_1>0$ such that (see \cite{alessandrini1998})
\begin{equation}\label{DH}
|D_{h_1}|\geq {1\over 2}|D|.
\end{equation}
where we set, for any $A \subset \mathbb{R}^d$ and $h>0$,
\begin{equation*}
A_{h}=\{x\in A:\;d(x,\pa A)>h \}.
\end{equation*}
\end{remark}
\subsection{Main results}
\
\par

Under the previous assumptions we consider the following boundary value problems. When the obstacle $D$ in $\Om$ is present,
the pair given by the velocity and the pressure of the fluid in $\Od$
is the weak solution $(u,p)\in (H^{2}(\Od))^{d}\times H^1(\Od)$ to
\begin{equation}\label{4}
\left\{\begin{array}{rllll}
-\Div(\sigma(u,p))&=&0& ,& \textrm{ in }\Od,\\
\Div u&=&0&,&\textrm{ in } \Od,\\
u&=&g&,& \textrm{ on }\partial\Omega,\\
u&=&0&,&\textrm{ on }\partial D.
\end{array}
\right.
\end{equation}
Then we can define the function $\psi$ by
\begin{equation}\label{2.6}
\D\psi=\sigma(u,p) n|_{\pa\Om}\in (H^{1/2}(\partial\Omega))^{d}
\end{equation}
and the quantity
\begin{equation*}
W=\D\int_{\pa\Om}(\sigma(u,p) n)\cdot u=\int_{\pa\Om}\psi\cdot g.
\end{equation*}
When the obstacle $D$ is absent, we shall denote by $(u_0,p_0)\in (H^2(\Om))^{d}\times H^1(\Om)$ the unique weak solution to the Dirichlet problem
\begin{equation}\label{5}
\left\{\begin{array}{rllll}
-\Div(\sigma(u_0,p_0))&=&0& ,& \textrm{ in }\Omega,\\
\Div u_0&=&0&,&\textrm{ in } \Omega,\\
u_0&=&g&,& \textrm{ on }\partial\Omega.\\
\end{array}
\right.
\end{equation}
Let us define
\begin{equation}\label{2.8}
\D\psi_0=\sigma(u_0,p_0) n|_{\pa\Om}\in (H^{1/2}(\partial\Omega))^{d},
\end{equation}
and
\begin{equation*}
W_0=\D\int_{\pa\Om}(\sigma(u_0,p_0) n)\cdot u_0=\int_{\pa\Om}\psi_0 \cdot g.
\end{equation*}
Our goal is to derive estimates of the size of $D$, $|D|$, in terms of $W$ and $W_0$.

\begin{theorem}\label{teo1}
Assume $\bf{(H1)}$, $\bf{(H2)}$, $\bf{(H4)}$ and $\bf{(H5)}$ . Then, we have
\begin{equation}\label{7}
|D|\leq\D K\left(\D\frac{W-W_0}{W_0}\right),
\end{equation}
where the constant $K$ depends on $\Omega, d, d_0, h_1, \rho_0, M_0, M_1$, and $\|g\|_{H^{1/2}(\partial\Omega)}/\|g\|_{L^{2}(\partial\Omega)}$.
\end{theorem}
\begin{theorem}\label{teo2}
Assume  $\bf{(H1)}$, $\bf{(H2)}$, $\bf{(H3)}$ and $\bf{(H4)}$.
Then, it holds
\begin{equation}\label{8}
\D C\frac{(W-W_0)^2}{\| g\|_{H^{3/2}(\pa\Om)}^2 W_{0}}\leq |D|,
\end{equation}
where $C>0$ depends on $M_1,\rho_0, d, d_0, \rho,  L$, and $Q$.
\end{theorem}
\begin {corollary}\label{corol1}
Assume $\bf{(H1)}$--$\bf{(H5)}$. Then, there exist two positive constant $K$ and $C$ as in \eqref{7} and \eqref{8} such that
\begin{equation}\label{7bis}
C\frac{(W-W_0)^2}{\| g\|_{H^{3/2}(\pa\Om)}^2 W_{0}}\leq |D|\leq K\left(\D\frac{W-W_0}{W_0}\right).
\end{equation}
\end{corollary}

\begin{remark}
We expect that a result similar to the one obtained in Corollary \ref{corol1} can be derived when we replace the Dirichet
boundary data with the condition
\begin{equation*}
\sigma(u,p)n = g, \quad {\rm on}\, \,  \partial \Omega,
\end{equation*}
$g$ satisfying suitable regularity assumptions and the compatibility condition
\begin{equation*}
\int_{\partial\Omega} g = 0.
\end{equation*}
\end{remark}

\setcounter{equation}{0}
\section{\bfseries Proofs of the main theorems}\label{sec3}
\
\par

The main idea of the proof of Theorem \ref{teo1} is an application of a three spheres inequality. In particular, we apply a result
contained in \cite{lin2010optimal} concerning the solutions to the following Stokes systems
\begin{equation}\label{9}
\left\{\begin{array}{rllll}
-\Delta u+A(x)\cdot\nabla u+B(x)u+\nabla p&=&0& ,& \textrm{ in }\Omega,\\
\Div u&=&0&,&\textrm{ in } \Omega.
\end{array}
\right.
\end{equation}
Then it holds:
\begin{theorem}[Theorem 1.1 \cite{lin2010optimal}]
Consider $0\leq R_0\leq 1$ satisfying $B_{R_0}(0)\subset\Om\subset\R^d$. Then, there exists a positive number $\tilde{R}<1$,
depending only on $d$, such that, if $0<R_1<R_2<R_3\leq R_0$ and $R_1/R_3<R_2/R_3<\tilde{R}$, we have
\begin{equation*}
\D\int_{|x|<R_2}|u|^2 dx\leq C\left(\int_{|x|<R_1}|u|^2 dx\right)^{\tau}\left(\int_{|x|<R_3}|u|^2 dx\right)^{1-\tau},
\end{equation*}
for $(u,p)\in(H^{1}(B_{R_0}(0)))^{d}\times H^{1}(B_{R_0}(0))$ solution to \eqref{9}. Here $C$ depends on $R_2/R_3$, $d$, and $\tau\in (0,1)$
depends on $R_1/R_3$, $R_2/R_3$, $d$. Moreover, for fixed $R_2$ and $R_3$, the exponent $\tau$ behaves like $1/(-\log R_1)$, when $R_1$ is sufficiently
small.
\end{theorem}
Based on this result, the following proposition holds:

\begin{proposition}[Lipschitz propagation of smallness, Proposition 3.1 \cite{ballerini2010stable}]\label{pro1}
Let $\Om$ satisfy {\rm ({\bf H1})} and $g$ satisfies {\rm ({\bf H4})}.
Let $u$ be a solution to the problem
\begin{equation}\label{10}
\left\{\begin{array}{rllll}
-\Div(\sigma(u,p))&=&0& ,& \textrm{ in }\Om,\\
\Div u&=&0&,&\textrm{ in } \Om,\\
u&=&g&,& \textrm{ on }\pa\Om.\\
\end{array}
\right.
\end{equation}
Then, there exists a constant $s>1$, depending only on $d$ and $M_0$, such that for every $r>0$ there exists a constant $C_{r}>0$, such that for every $x\in \Omega_{s r}$, we have
\begin{equation}\label{11}
\D\int_{B_{r}(x)}|\nabla u|^{2}dx\geq C_{r}\int_{\Om}|\nabla u|^{2}dx,
\end{equation}
where the constant $C_{r}>0$ depends only on $d,M_0,M_1,\rho_0,r, \D\frac{\|g\|_{H^{1/2}(\pa\Om)}}{\|g\|_{L^{2}(\pa\Om)}}$.
\end{proposition}

\smallskip

Following the ideas developed in \cite{alessandrini2002detecting}, we establish a key variational inequality relating the boundary data
$W-W_0$ to the $L^2$ norm of the gradient of $u_0$ inside the cavity $D$.
\begin{lemma}\label{l1}
Let $u_0\in (H^{1}(\Omega))^{d}$ be the solution to problem \eqref{5} and $u\in (H^{1}(\Od))^{d}$ be the solution to problem \eqref{4}.
Then, there exists a positive constant $C=C(\Om)$ such that
\begin{equation}\label{12}
\D\int_{D}|\nabla u_0|^{2}\leq C(W - W_0)=C \int_{\partial D}u_0\cdot\sigma(u,p)n,
\end{equation}
where $n$ denotes the exterior unit normal to $\pa D$.
\end{lemma}
\begin{proof}
Let $(u,p)$ and $(u_0,p_0)$ be the solutions to problems \eqref{4} and \eqref{5}, respectively.
We multiply the first equation of \eqref{4} by $u_0$ and after integrating by parts, we have
\begin{equation}\label{3.5}
\D\int_{\Od}\sigma(u,p)\cdot\nabla u_0-\int_{\partial\Omega}(\sigma(u,p) n)\cdot u_0+\int_{\partial D}(\sigma(u,p) n)\cdot u_0=0,
\end{equation}
where $n$ denotes either the exterior unit normal to $\pa\Om$ or to $\pa D$.

In a similar way, multiplying the first equation of \eqref{5} by $u_0$, we obtain
\begin{equation}\label{3.6}
\D\int_{\Omega}\sigma(u_0,p_0)\cdot\nabla u_0-\int_{\partial\Omega}(\sigma(u_0,p_0) n)\cdot u_0=0.
\end{equation}
Now, replacing $\psi$ and $\psi_0$ into the equations \eqref{3.5}-\eqref{3.6}, we get
\begin{equation}\label{13}
\left\{\begin{array}{r}
\D\int_{\Od}\sigma(u,p)\cdot\nabla u_0-\int_{\partial\Omega}\psi\cdot g+\int_{\partial D}(\sigma(u,p) n)\cdot u_0=0,\\
\D\int_{\Omega}\sigma(u_0,p_0)\cdot\nabla u_0-\int_{\partial\Omega}\psi_0\cdot g=0.
\end{array}
\right.
\end{equation}
Let us define
\begin{equation*}
\tilde{u}(x) = \left\{
\begin{array}{rl}
u&\textrm{ if }x\in\Od,\\
0&\textrm{ if }x\in\overline{D}.
\end{array}\right.
\end{equation*}
Since $u=0$ on $\pa D$, we have $\tilde{u}\in (H^1(\Om))^{d}$.
So, multiplying \eqref{4} and \eqref{5} by $\tilde{u}$, we obtain
\begin{equation}\label{15}
\left\{\begin{array}{r}
\D\int_{\Od}\sigma(u,p)\cdot\nabla \tilde{u}-\int_{\partial\Omega}\psi\cdot g+\underbrace{\int_{\partial D}(\sigma(u,p) n)\cdot\tilde{u}}_{=0}=0,\\
\D\int_{\Od}\sigma(u_0,p_0)\cdot\nabla\tilde{u}-\int_{\partial\Omega}\psi_0\cdot g=0.
\end{array}
\right.
\end{equation}
Using the definition of $\sigma(u,p)$ in the first equation of \eqref{13}, we have
\begin{eqnarray*}
\label{3.9}
0=\D\int_{\Od}\sigma(u,p)\cdot\nabla u_0-\int_{\partial\Omega}\psi\cdot g+\int_{\partial D}(\sigma(u,p) n)\cdot u_0\\
=\D\int_{\Od}(2e(u)-pI)\cdot\nabla u_0-\int_{\partial\Omega}\psi\cdot g+\int_{\partial D}(\sigma(u,p) n)\cdot u_0\\
=\D\int_{\Od}2e(u)\cdot\nabla u_0-\int_{\Od}p(\Div \ u_0)-\int_{\partial\Omega}\psi\cdot g+\int_{\partial D}(\sigma(u,p) n)\cdot u_0\\
=\D\int_{\Od}2e(u)\cdot\nabla u_0-\int_{\partial\Omega}\psi\cdot g+\int_{\partial D}(\sigma(u,p) n)\cdot u_0,
\end{eqnarray*}
where we use the fact that $\Div \, u_0=0$.
For the next step, we need a different expression for the term $e(u)\cdot\nabla u_0$.
We claim that, for every $v\in (H^1(\Om))^{d}$ such that $\Div \, v=0$, we have $e(u)\cdot\nabla v=e(u)e(v)$. Indeed,
\begin{equation*}
\begin{array}{rl}
2e(u)\cdot\nabla v&\D =\left(\frac{\pa u_{i}}{\pa x_{j}}+\frac{\pa u_{j}}{\pa x_{i}}\right)\frac{\pa v_{i}}{\pa x_{j}}\\[0.4em]
&\D =\frac{1}{2}\left(\frac{\pa u_{i}}{\pa x_{j}}+\frac{\pa u_{j}}{\pa x_{i}}\right)\frac{\pa v_{i}}{\pa x_{j}}
+\frac{1}{2}\left(\frac{\pa u_{i}}{\pa x_{j}}+\frac{\pa u_{j}}{\pa x_{i}}\right)\frac{\pa v_{j}}{\pa x_{i}}\\[0.6em]
&\D =e(u)\cdot\nabla v+e(u)\cdot\nabla v^{T}=2e(u)\cdot e(v).
\end{array}
\end{equation*}
Therefore, equalities \eqref{13} and \eqref{15} can be rewritten as
\begin{eqnarray}
\D2\int_{\Od}e(u)\cdot e(u_0)-\int_{\partial\Omega}\psi\cdot g+\int_{\partial D}u_0\cdot(\sigma(u,p) n)=0,\label{18}\\
\D2\int_{\Omega}|e(u_0)|^{2}-\int_{\partial\Omega}\psi_0\cdot g=0,\label{19}\\
\D2\int_{\Od}|e(u)|^{2}-\int_{\partial\Omega}\psi\cdot g=0,\label{16}\\
\D2\int_{\Od}e(u_0)\cdot e(u)-\int_{\partial\Omega}\psi_0\cdot g=0.\label{17}
\end{eqnarray}
We note that if we subtract \eqref{17} from \eqref{18} we get
\begin{equation}\label{20}
\D\int_{\partial\Omega}(\psi-\psi_0)\cdot g=\int_{\partial D}u_0\cdot(\sigma(u,p) n).
\end{equation}
Now, let us consider the quadratic form
\begin{equation*}
\begin{array}{rl}
\D\int_{\Omega}e(\tilde{u}-u_0)\cdot e(\tilde{u}-u_0)& \D =\int_{\Omega}|e(u_0)|^{2}+\int_{\Od}|e(u)|^{2}-2\int_{\Od}e(u)\cdot e(u_0)\\
&\D =\frac{1}{2}\int_{\partial\Omega}\psi_0\cdot g+\frac{1}{2}\int_{\partial\Omega}\psi\cdot g-\int_{\partial\Omega}\psi_0\cdot g\\
&\D =\frac{1}{2}\int_{\partial\Omega}(\psi-\psi_0)\cdot g.
\end{array}
\end{equation*}
By Korn's inequality there exists a constant $C=C(\Om) >0,$ such that
\begin{equation*}
\D\int_{\Om}|\nabla (\tilde{u}-u_0)|^2\leq C\int_{\Om}|e(\tilde{u}-u_0)|^2.
\end{equation*}
Finally, by the chain of inequalities
\begin{equation*}
\begin{array}{rl}
&\displaystyle\D\int_{D}|\nabla u_0|^{2}=\int_{D}|\nabla(\tilde{u}-u_0)|^2 \displaystyle\leq\int_{\Om}|\nabla(\tilde{u}-u_0)|^2 \\
&\displaystyle\leq C\int_{\Om}|e(\tilde{u}-u_0)|^2=C\int_{\partial\Omega}(\psi-\psi_0)\cdot g = C(W-W_0),
\end{array}
\end{equation*}
and \eqref{20} the claim follows.
\end{proof}
Now, using the previous results, we are able to prove Theorem \ref{teo1}.
\begin{proof}
The proof is based on arguments similar to those used in \cite{alessandrini2002detecting} and \cite{alessandrini2004detecting}.
Let us consider the intermediate domain $\Om_{d_{0}/2}$. Recalling that $d(D,\pa\Om)\geq d_0$, we have $d(D,\pa\Om_{d_{0}/2})\geq\frac{d_0}{2}.$
Let $\epsilon=\min\left(\frac{d_0}{2},\frac{h_1}{\sqrt{d}}\right) >0$. Let us cover the domain $D_{h_1}$ with cubes $Q_{l}$ of side $\epsilon$,
for $l=1,\ldots,N$. By the choice of $\epsilon$, the cubes $Q_{l}$ are contained in $D$. Then,
\begin{equation}\label{22}
\D\int_{D}|\nabla u_{0}|^{2}\geq\int_{\cup_{l=1}^{N}Q_{l}}|\nabla u_0|^{2}\geq\frac{|D_{h_1}|}{\epsilon^{d}}\int_{Q_{\overline{l}}}|\nabla u_0|^{2},
\end{equation}
where $\overline{l}$ is chosen in such way that
\begin{equation*}
\D\int_{Q_{\overline{l}}}|\nabla u_0|^{2}=\min_{l}\int_{Q_{l}}|\nabla u_0|^{2}>0.
\end{equation*}
We observe that the previous minimum is strictly positive because, if not, then $u_0$ would be constant in $Q_{\overline l}$.
Thus, from the unique continuation property, $u_0$ would be constant in $\Om$ and since there exists a point $P\in\pa\Om,$ such that,
\begin{equation*}
g=0\textrm{ on }\pa\Om\cap B_{\rho_0}(P),
\end{equation*}
we would have that $u_0 \equiv 0$ in $\Om$, contradicting the fact that $g$ is different from zero. Then, the minimum is strictly positive.

Let $\overline{x}$ be the center of $Q_{\overline{l}}$. From the estimate \eqref{11} in Proposition \ref{pro1} with $x=\overline{x}$,
$r=\frac{\epsilon}{2}$, we deduce
\begin{equation}\label{3.17}
\D\int_{Q_{\overline{l}}}|\nabla u_0|^2 \geq C\int_{\Om}|\nabla u_0|^2.
\end{equation}
On account of Remark \ref{R1bis}, we obtain
\begin{equation}\label{3.18}
\D\int_{D}|\nabla u_0|^2 \geq \D\frac{\frac{1}{2}|D|}{\epsilon^{d}}C\int_{\Om}|\nabla u_0|^2=|D| {C}\int_{\Om}|\nabla u_0|^2.
\end{equation}
We estimate the right hand side of \eqref{3.18}. First, using \eqref{19} we have
\begin{eqnarray}\label{3.19}
\D\int_{\pa\Om}\psi_0\cdot g&=2\int_{\Om}|e(u_0)|^2=2\int_{\Om}\frac{|\nabla u_0+\nabla u_0^{T}|^2}{4}\\
&=2\left(\int_{\Om}\frac{|\nabla u_0|^2+|\nabla u_0^{T}|^2+2\nabla u_0 \cdot\nabla u_0^{T}}{4}\right)
\end{eqnarray}
Now, H\"older's inequality implies
\begin{equation}\label{3.20}
\D\int_{\pa\Om}\psi_0\cdot g\leq 2\int_{\Om}|\nabla u_0|^2.
\end{equation}
Then, coming back to \eqref{3.18}, we obtain that there exists a constant $K$, depending on $\Omega,d, d_0, h_1, rho_0, M_0, M_1$,
and $\|g\|_{H^{1/2}(\partial\Omega)}/\|g\|_{L^{2}(\partial\Omega)}$ such that
\begin{equation}\label{3.21}
\D\int_{D}|\nabla u_0|^2 \geq |D| K\int_{\pa\Om}\psi_0\cdot g.
\end{equation}
Combining \eqref{3.21} and Lemma \ref{l1} we have
\begin{equation}\label{3.22}
C\D\int_{\pa\Om}(\psi-\psi_0)\cdot g \geq \int_{D}|\nabla u_0|^2\geq \left({K}\int_{\pa\Om}\psi_0\cdot g\right)|D|.
\end{equation}
Therefore, we can conclude that
\begin{equation*}
|D|\leq \D {K}\frac{W-W_0}{W_0},
\end{equation*}
where $\tilde{K}$ is a constant depending on $\Omega,d, d_0, h_1, \rho_0,M_0, M_1$, and $\D\frac{\|g\|_{H^{1/2}(\partial\Omega)}}{\|g\|_{L^{2}(\partial\Omega)}}$.
\end{proof}

In order to prove Theorem \ref{teo2}, we make use of the following Poincar\'{e} type inequality.
\begin{proposition}[Proposition $3.2$ \cite{alessandrini2002detecting}]\label{pro2}
Let $D$ be a bounded domain in $\R^d$ of class $C^{2,\alpha}$ with constants $\rho,L$ and such that \eqref{3} holds.
Then, for every $u\in (H^{1}(D))^{d}$ we have
\begin{equation}\label{23}
\D\int_{\pa D}|u-\overline{u}|^{2}\leq \overline{C}\rho\int_{D}|\nabla u|^{2},
\end{equation}
where $\overline{u}=\D\frac{1}{|\pa D|}\int_{\pa D}u$
and the constant $\overline{C}>0$ depends only on $L,Q$.
\end{proposition}
Using this result and Lemma \ref{l1} we can prove now Theorem \ref{teo2}.
\begin{proof}
Let $\overline{u}_0$ be the following number
\begin{equation}\label{3.25}
\overline{u}_0=\D\frac{1}{|\pa D|}\int_{\pa D}u_0.
\end{equation}
Then, we deduce that
\begin{equation}\label{3.26}
\D\int_{\pa D}(\sigma(u,p) n)\cdot u_0=\int_{\pa D}(\sigma(u,p) n)\cdot u_0 - \int_{\pa D}(\sigma(u,p) n)\cdot\overline{u}_0,
\end{equation}
because $\int_{\pa D}\sigma(u,p)\cdot n=0$. From equality \eqref{20} in Lemma \ref{l1}, we have
\begin{equation}\label{3.27}
W-W_0=\D\int_{\pa D}(\sigma(u,p) n)\cdot u_0=\int_{\pa D}(\sigma(u,p) n)\cdot (u_0-\overline{u}_0).
\end{equation}
Applying H\"older inequality in the right hand side of \eqref{3.27} we obtain
\begin{equation}\label{3.28}
W-W_0\leq \left(\int_{\pa D}|u_0-\overline{u_0}|^2\right)^{1/2}\left(\int_{\pa D}|\sigma(u,p) n|^2\right)^{1/2}.
\end{equation}
Now, using Poincar\'e inequality \eqref{23} in the first integral on the right hand side of \eqref{3.28}, we get
\begin{equation}\label{3.29}
W-W_0\leq C\left(\int_{D}|\nabla u_0|^2\right)^{1/2}\left(\int_{\pa D}|\sigma(u,p) n|^2\right)^{1/2},
\end{equation}
where $C>0$ depends on $|\Om|, Q, \rho$ and $L$.
The first integral on the right hand side of \eqref{3.29} can be estimated as
\begin{equation}\label{3.30}
\D\int_{D}|\nabla u_0|^2 \leq |D|^{1/2}\sup_{D}|\nabla u_0|.
\end{equation}

Now, we need to give an interior estimate for the gradient of $u_0$. For this, we observe that for the regularity of the Stokes problem we
have $u_0\in (H^2(\Om))^d$. Then, we may take the Laplacian of the second equation in \eqref{5}
\begin{equation*}
\Delta \Div \ u_0=0.
\end{equation*}
Therefore, commuting the differential operators, we obtain that the pressure is an harmonic function. This implies that each component of
$u_0$ is a biharmonic function.
Then, using interior regularity estimates for fourth order equations, we deduce that
\begin{equation}\label{3.31}
\D\sup_{D}|\nabla u_0|\leq C\|u_0\|_{L^2(\Om)},
\end{equation}
where the constant $C$ depends on $Q$, $|\Om|$ and $d_0$.  Estimate (\ref{3.31}) can be obtained considering the following results.
We know that the embedding from $H^4(\Om)$ to $C^{k}(\Om)$ is continuous for $0\leq k < 4-\frac{d}{2}$, with $d=2,3$. Then, in particular,
\begin{equation*}
\|u_0\|_{C^1(D)}\leq C \|u_0\|_{H^4(D)}.
\end{equation*}
Moreover, from the interior regularity of fourth order equations, see  \cite[Th. 8.3]{morassi2007size}, we obtain
\begin{equation*}
\|u_0\|_{H^4(D)}\leq C\|u_0\|_{H^2(\Om_{d_0/2})}.
\end{equation*}
Finally, considering the estimates in \cite{auscher2000equivalence} and \cite{boyer2012mathematical}, we have
\begin{equation*}
\|u_0\|_{H^2(\Om_{d_0/2})}\leq C\|u_0\|_{L^2(\Om_{d_0/4})}\leq C\|u_0\|_{L^2(\Om)},
\end{equation*}
and \eqref{3.31} holds. We refer to \cite{auscher2000equivalence,barton2014gradient,cordes1956erste}, and references therein, for more details on interior estimates for elliptic operators.

As the boundary data $g$ satisfies $\bf{(H4)}$, we use the classical Poincar\'e inequality and obtain
\begin{equation}\label{3.33}
\D \|u_0\|_{L^2(\Om)}\leq C\|\nabla u_0\|_{L^2(\Om)}.
\end{equation}
Therefore, by means of the inequality $\displaystyle\int_{\Om}|\nabla u_0|^2\leq C\displaystyle \int_{\pa\Om}\psi_0\cdot g$, we deduce
\begin{equation}
\D\left(\int_{D}|\nabla u_0|^2\right)^{1/2}\leq C |D|^{1/2} W_0^{1/2}.
\end{equation}
Now, concerning the second integral in \eqref{3.29} we note that from the Trace Theorem it follows
\begin{equation}\label{26}
\| \sigma(u,p)\cdot n\|_{L^2(\pa D)}\leq C (\| u\|_{H^2(\Od)}+\| p\|_{L^2(\Od)}),
\end{equation}
and applying Theorem \ref{TSR} we obtain the inequality
\begin{equation}\label{27}
\| \sigma(u,p)\cdot n\|_{L^2(\pa\Om)}\leq C(\| u\|_{H^2(\Od)}+\| p\|_{L^2(\Od)})\leq C \| g\|_{H^{3/2}(\pa\Om)}.
\end{equation}
Therefore, it holds
\begin{equation*}
C\D\frac{(W-W_0)^2}{\| g\|_{H^{3/2}(\pa\Om)}^2W_0}\leq |D|,
\end{equation*}
where $C$ depends on $M_1, \rho_0, d, \rho, L$ and Q. This completes the proof.
\end{proof}

We conclude the section observing that proof of Corollary \ref{corol1} is a straightforward consequence of Theorem \ref{teo1} and Theorem \ref{teo2}.

\section{\bfseries Computational examples}\label{sec4}
\
\par

In this section we will perform some numerical experiments to compute $|\frac{W-W_0}{W_0}|$ for classes of cavities for which our result holds.
In particular, we expect to collect numerical evidence that the ratio between $\frac{|D|}{|\Omega|}$ and $|\frac{W-W_0}{W_0}|$ is bounded from below and above 
by two constants indicating that, due to the limits of our technique, the estimate from below is not optimal. 
Indeed, the numerical experiments we perform give some preliminary indications that this conjecture is true.

Moreover, we are interested in studying the dependence of this ratio on $d_0$, which bounds from below the distance of $D$ from $\partial\Omega$, and  
the size of the inclusions. 

A more systematic analysis would require the knowledge of explicit solutions $u$ and $u_0$. This would allow to compute analytically the constants in the upper
and lower bounds, at least for some particular geometries.  On the contrary to the case in \cite{alessandrini2002detecting}, for
the Stokes system it is difficult to find explicit solutions.

For the experiments we use the free software \emph{FreeFem++} (see \cite{MR3043640}).  Moreover, in all numerical tests we consider a square  domain $\Om$, discretized with a mesh of $100\times 100$ elements, and with boundary condition $u|_{\pa\Om}=g$ as in Figure $4.1$. The datum $g$ 
satisfies the assumptions ${\bf (H4)}$ and ${\bf (H5)}$.

\begin{figure}[H]
\begin{center}
\includegraphics[width=0.8\textwidth]{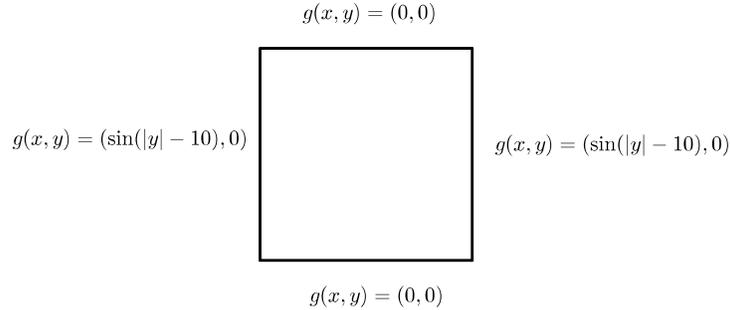}
\caption{Square domain in $2$-D with boundary condition $g$.}
\end{center}
\end{figure}

The first series of numerical tests has been performed by varying the position and the size of a circle inclusion $D$ with volume up to $8\%$ of the total size of the domain. 
In particular, we consider a circle inclusion with volume $0.2\%$, $3.1\%$ and $7.1\%$ with respect to $|\Omega|$. We have placed these circles in eight different positions, see Figure $4.2$. 
The results are collected in Figure $4.3$, $4.4$ and $4.5$, for different values of the distance $d_0$ between the object $D$ and the boundary of $\Om$. 
Also, the averages of all this simulations are collected in Figure $4.6$.

\begin{figure}[H]
\centering
\includegraphics[width=0.4\textwidth]{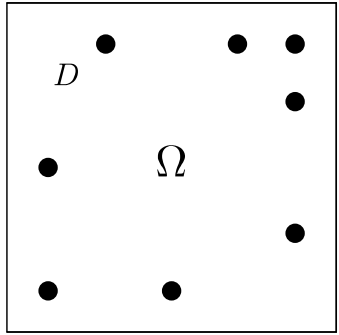}
\caption{The eight positions of the circle inclusion $D$.}
\end{figure}

In order to compare our numerical results with the theoretical upper and lower bounds \eqref{7} and \eqref{8}, it is interesting to study the relationship between $\frac{|D|}{|\Om|}$ and $|\frac{W-W_0}{W_0}|$. 
As we expected from the theory, the points $(\frac{W-W_0}{W_0},\frac{|D|}{|\Om|})$ are confined inside an angular sector delimited by two straight lines.

\begin{figure}[H]
\centering
\includegraphics[width=0.65\textwidth]{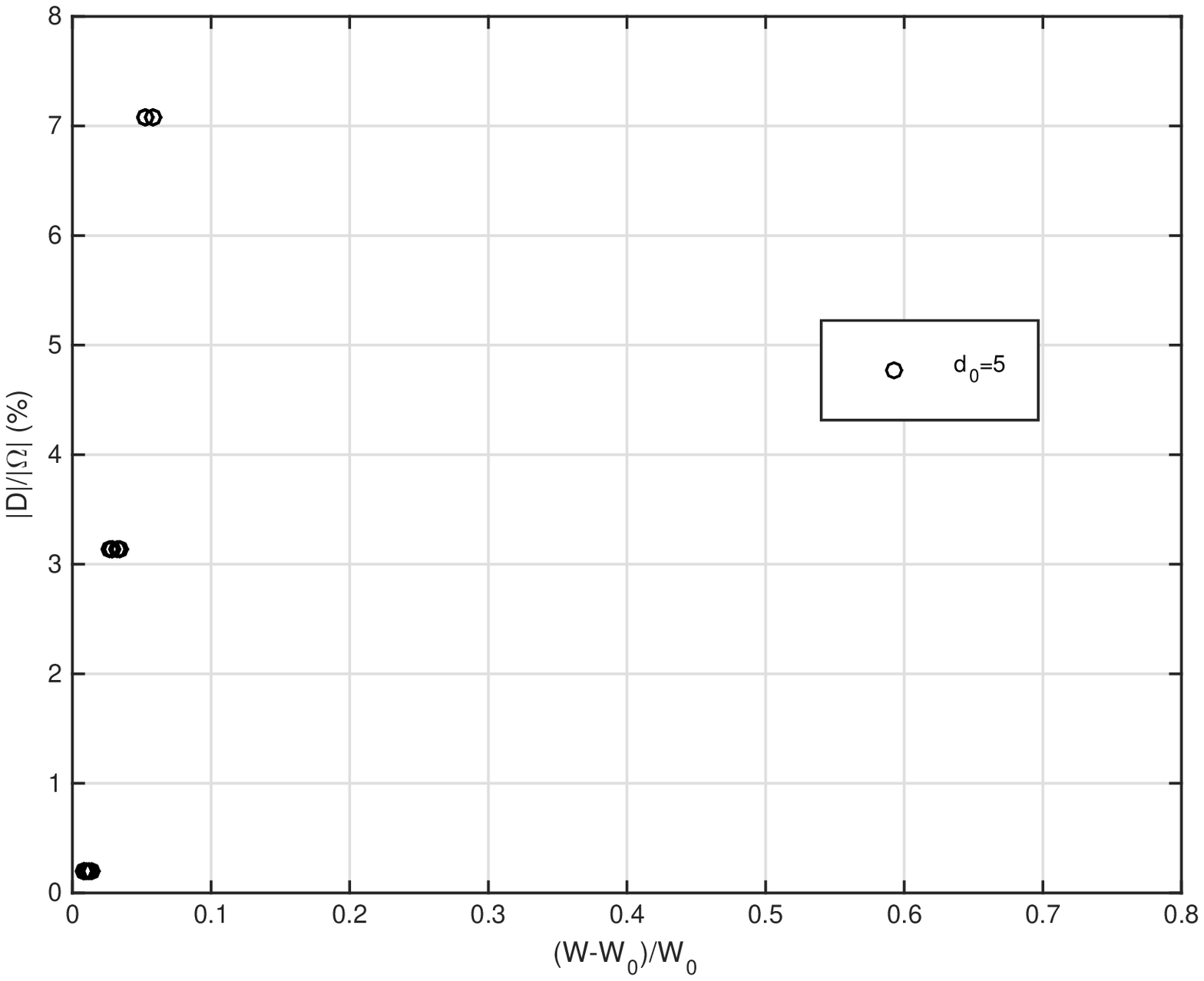}
\caption{Case $d_0=5$ for circle inclusion.}
\end{figure}

\begin{figure}[H]
\centering
\includegraphics[width=0.65\textwidth]{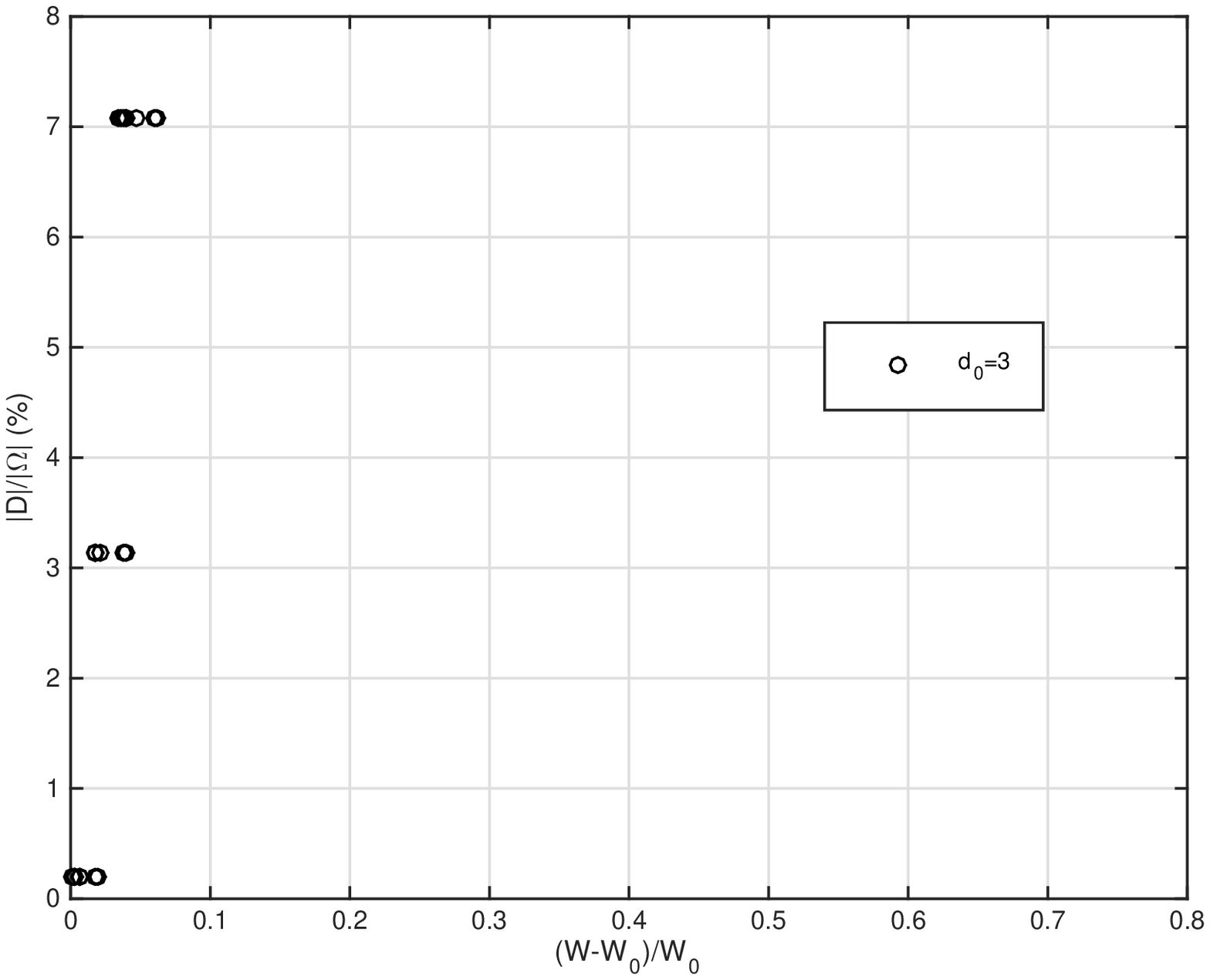}
\caption{Case $d_0=3$ for circle inclusion.}
\end{figure}

\begin{figure}[H]
\centering
\includegraphics[width=0.65\textwidth]{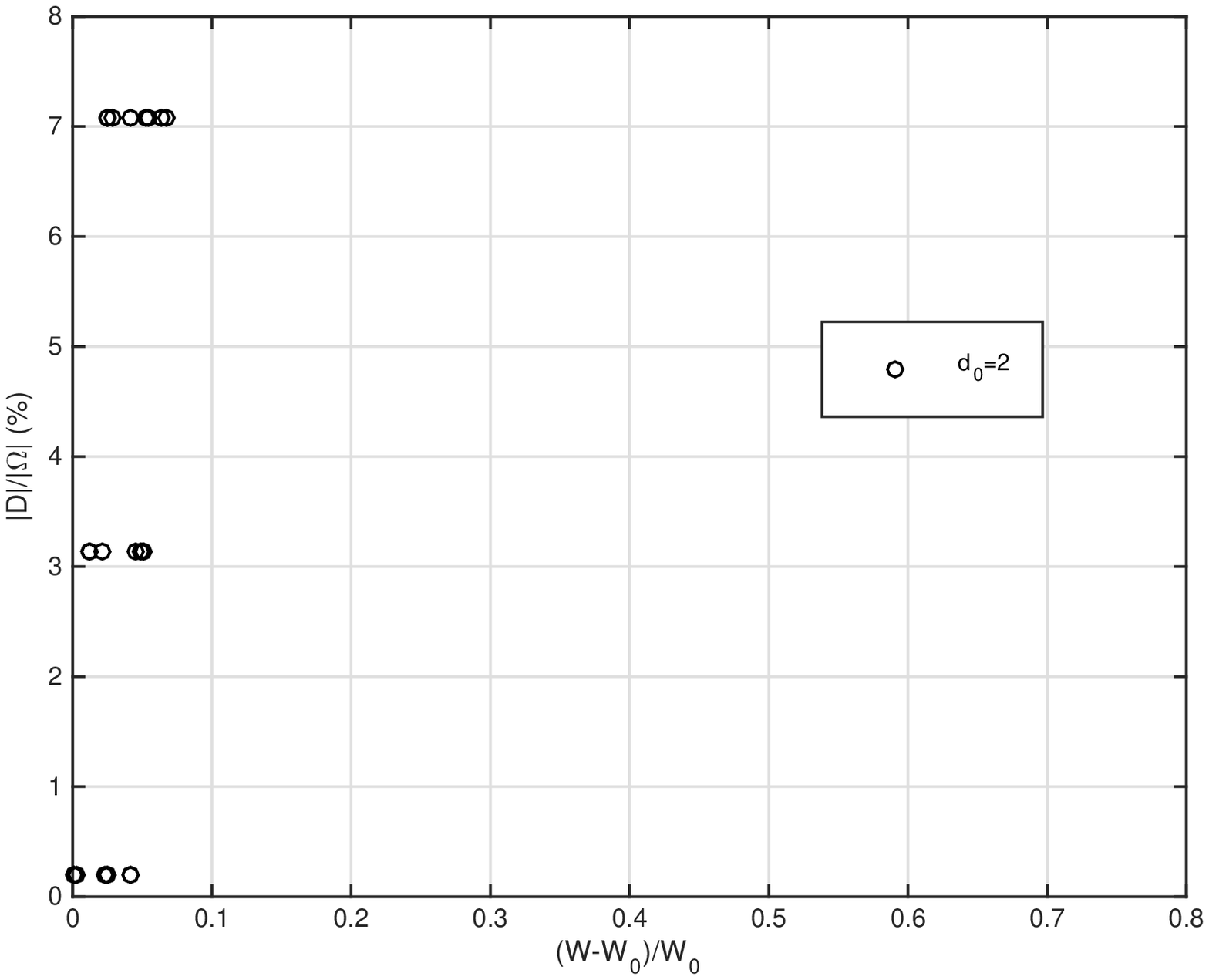}
\caption{Case $d_0=2$ for circle inclusion.}
\end{figure}

\begin{figure}[H]
\centering
\includegraphics[width=0.65\textwidth]{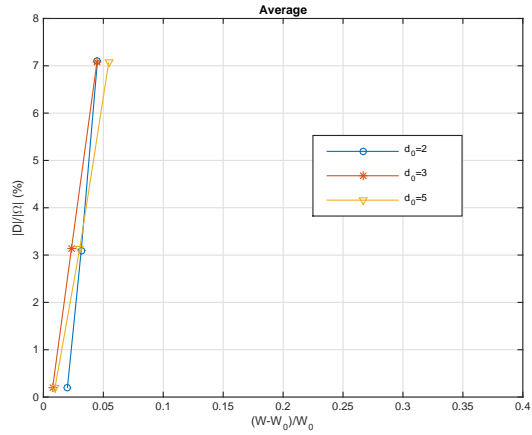}
\caption{Averages of the ratio $\frac{W-W_0}{W_0}$ with different $d_0$ for circle inclusion.}
\end{figure}

However, it is quite clear that when $d_0$ decreases, then the lower bound becomes worse. To illustrate this situation, we simulate also the case when the distance is $d_0=1$, 
see Figure $4.7$.

\begin{figure}[H]
\centering
\includegraphics[width=0.65\textwidth]{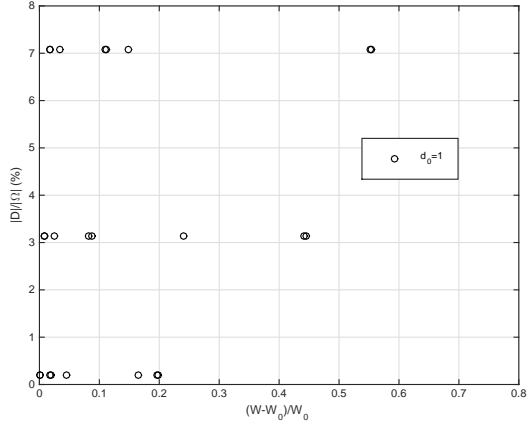}
\caption{Case $d_0=1$ for circle inclusion.}
\end{figure}

As a second class of experiments, we consider what happens when the size of the circle increases. 
In this case we can observe that the number $|\frac{W-W_0}{W_0}|$ grows rapidly when the volume occupies almost the entire domain. The result is collected in Figure $4.8$.

Again it is observed the relationship between the volume of the object with the quotient $(W-W_0)/W_0$. This gives us an indication that the estimates found in Theorems \ref{teo1} and \ref{teo2} 
involve constants that do not depend on the inclusion.

\begin{figure}[H]
\centering
\includegraphics[width=0.65\textwidth]{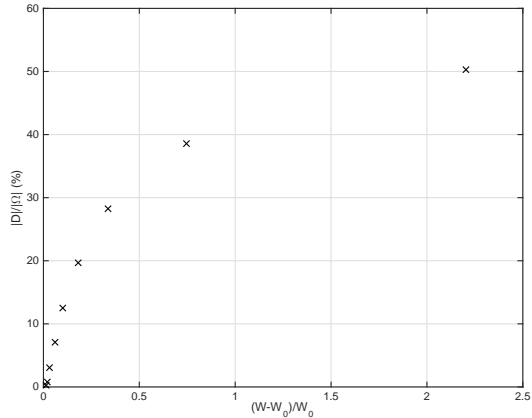}
\caption{Influence of the size of the circle.}
\end{figure}

\begin{remark}
From the previous analysis an interesting problem would be to find optimal lower and upper bounds for this model.

An other interesting issue would be to weaken the a-priori assumptions imposed on the obstacle, as for example the 
fatness condition (see, for instance, \cite{cristo2013size}, where this restriction is removed in the case of the shallow shell equations).

\end{remark}

\bigskip

\noindent \textbf{Acknowledgements.}  This work was partially supported by PFB03-CMM and Fondecyt 1111012.
The work of E. Beretta was supported by GNAMPA (Gruppo Nazionale per l'Analisi Matematica, la Probabilit\`a e le loro Applicazioni) of INdAM
(Istituto Nazionale di Alta Matematica) and part of it was done while the author was visiting New York University Abu Dhabi.
The work of C. Cavaterra was supported by the FP7-IDEAS-ERC-StG \#256872 (EntroPhase) and by GNAMPA (Gruppo Nazionale per l'Analisi Matematica,
la Probabilit\`a e le loro Applicazioni) of INdAM.
Part of this work was done while J. Ortega was visiting the Departamento de Matem\'atica, Universidad Aut\'onoma de Madrid - UAM and the Instituto de Ciencias
Matem\'aticas ICMAT-CSIC, Madrid, Spain.
The work of S. Zamorano was supported by CONICYT-Doctorado nacional 2012-21120662. Part of this work was done while S. Zamorano was visiting the 
Basque Center for Applied Mathematics and was partially supported by the Advanced Grant NUMERIWAVES/FP7-246775 of the European Research Council Executive Agency, 
the FA9550-15-1-0027 of AFOSR, the MTM2011-29306 and MTM2014-52347 Grants of the MINECO.


\end{document}